\documentclass[12pt]{amsart}
\usepackage{amssymb, amsthm, amsmath, amsfonts,amscd}
\usepackage{graphics}
\usepackage{hyperref}
\usepackage[all]{xy}
\usepackage{enumerate}
\usepackage[mathscr]{eucal}
\usepackage{bbm}
\usepackage{tikz}

\providecommand{\U}[1]{\protect\rule{.1in}{.1in}}
\def\theenumi{\arabic{enumi}}

\def\theenumii{\alph{enumii}}
\def\p@enumii{\theenumi.}

\def\theenumiii{\arabic{enumiii}}
\def\p@enumiii{(\theenumi)(\theenumii)}

\def\p@enumiv{\p@enumiii.\theenumiii}
\theoremstyle{plain}
\newtheorem{theorem}{Theorem}[section]
\newtheorem{lemma}[theorem]{Lemma}
\newtheorem{proposition}[theorem]{Proposition}

\newtheorem{corollary}[theorem]{Corollary}
\newtheorem{conjecture}[theorem]{Conjecture}
\numberwithin{equation}{section}

\theoremstyle{definition}
\newtheorem{question}{Question}
\newtheorem{definition}[theorem]{Definition}

\newtheorem{remark}[theorem]{Remark}

\newtheorem{thmab}{Theorem}
\newtheorem{corab}[thmab]{Corollary}

\renewenvironment{proof}[1][\proofname]{{\bfseries #1\\}}{\qed}

\DeclareMathOperator{\Conf}{Conf}
\DeclareMathOperator{\UConf}{UConf}

\DeclareMathOperator{\rank}{rank}

\newcommand{\K}{{\mathcal{K}}}
\newcommand{\R}{\mathbb{R}}

\newcommand{\Z}{{\mathbb{Z}}}
\newcommand{\Q}{\mathbb{Q}}
\newcommand{\dt}{\bullet}
\newcommand{\FI}{{\mathscr{FI}}}
\newcommand{\Sn}{\mathfrak{S}}
\newcommand{\M}{\mathcal{M}}
\newcommand{\N}{\mathcal{N}}
\newcommand{\arXiv}[1]{\href{http://arxiv.org/abs/#1}{\nolinkurl{arXiv:#1}}}
\newcommand{\arXivV}[2]{\href{http://arxiv.org/abs/#1}{\nolinkurl{arXiv:#1v#2}}}

\title{Stability Phenomena in the Homology of Tree Braid Groups} 

\author{Eric Ramos}
\address{Department of Mathematics, University of Wisconsin - Madison.}
\email{eramos@math.wisc.edu}

\thanks{The author was supported by NSF-RTG grant DMS-1502553.}

\begin{document}
\maketitle
\begin{abstract}
For a tree $G$, we study the changing behaviors in the homology groups $H_i(B_nG)$ as $n$ varies, where $B_nG := \pi_1(\UConf_n(G))$. We prove that the ranks of these homologies can be described by a single polynomial for all $n$, and construct this polynomial explicitly in terms of invariants of the tree $G$. To accomplish this we prove that the group $\bigoplus_n H_i(B_nG)$ can be endowed with the structure of a finitely generated graded module over an integral polynomial ring, and further prove that it naturally decomposes as a direct sum of graded shifts of squarefree monomial ideals. Following this, we spend time considering how our methods might be generalized to braid groups of arbitrary graphs, and make various conjectures in this direction.
\end{abstract}

\section{Introduction}

\subsection{Introductory remarks and statements of the main theorems}

In recent years there has been a push towards understanding the mechanisms connecting various well-known asymptotic stability results in topology and algebra. For instance, let $\N$ be a connected oriented manifold of dimension $\geq 2$, which is the interior of a manifold with boundary, and write $\Conf_n(\N)$ for the $n$-strand configuration space
\[
\Conf_n(\N) := \{(x_1,\ldots,x_n) \in \N^n \mid x_i \neq x_j\}.
\]
There is a natural action on $\Conf_n(\N)$ by the symmetric group $\Sn_n$, and we may therefore define the $n$-strand unordered configuration space
\[
\UConf_n(\N) := \Conf_n(\N)/\Sn_n.
\]
A classical theorem of McDuff \cite[Theorem 1.2]{McD} implies that for each $i$, and any $n \gg i$, the group $H_i(\UConf_n(\N))$ is independent of $n$. In contrast to the work of McDuff, it can be seen that the analogous statement is not true for the ordered configuration spaces $\Conf_n(\N)$. What is true, however, is perhaps the next best thing. It follows from work of Church, Ellenberg, and Farb \cite[Theorem 6.4.3]{CEF} that for any $i$, there is a polynomial $P \in \Q[x]$ such that the Betti number $\dim_\Q(H_i(\Conf_n(U);\Q))$ agrees with $P(n)$ for $n \geq 0$. Results of this type fall under the heading of what one might call asymptotic algebra.\\

The modern philosophy of asymptotic algebra can be roughly stated as follows: a family of algebraic objects which display asymptotic stability phenomena can often times be encoded into a single object, which is finitely generated in an appropriate abelian category. In the case of McDuff, for each $i$ and $n$ the group $H_i(\UConf_n(\N))$ can be realized as the $n$-graded piece of some finitely generated graded module over $\Z[x]$. The result of Church, Ellenberg, and Farb involves showing that the $\Sn_n$-representations $H_i(\Conf_n(U);\Q)$ are each constituents of some finitely generated representation of the category $\FI$, of finite sets and injections (see \cite{CEF} for more on this notion). This philosophy is also heavily featured in Sam and Snowden's recent resolution of Stembridge's conjecture \cite{SS}. The goal of this paper is to apply similar techniques to the homologies of the unordered configuration spaces of trees. Note that this problem was considered by L\"utgehetmann in his Master's thesis \cite{L}. The results of that work are disjoint from the current work.\\

In this paper, a \textbf{graph} will always refer to a connected, compact CW-complex of dimension 1. A \textbf{tree} is a graph which is contractible as a topological space. An \textbf{essential vertex} of $G$ is a vertex of degree, or valency, at least 3, while an \textbf{essential edge} of $G$ is a connected component of the space obtained by removing all essential vertices from $G$. Note that both the essential edges and vertices of a graph are unaffected by subdivision of edges, and can be thought of as the topologically essential pieces of the graph.\\

Our main result relates to the asymptotic behavior of the homologies of the braid group of a tree. To state this result, we first need to define a kind of connectivity invariant for trees.\\

Let $G$ be a tree. Then we set
\[
\Delta_G^i := \max_{\{\{v_{j_1},\ldots,v_{j_i}\} \mid v_{j_k}\text{ essential}\}} \{ \dim_\Q(H_0(G - \{v_{j_1},\ldots,v_{j_i}\};\Q))\}.
\]
In words, $\Delta_G^i$ is the maximum number of connected components that $G$ can be broken into by removing exactly $i$ essential vertices. Therefore $\Delta_G^1$ is the maximum degree of a vertex in $G$, while, if $G$ has $N_G$ essential vertices, $\Delta_G^{N_G}$ is the number of essential edges of $G$. By convention, $\Delta_G^0 = 1$, while $\Delta_G^i = 0$ for $i > N_G$.\\

\begin{thmab}\label{bettipolyrefine}
Let $G$ be a tree, and write $B_nG$ to denote the braid group $B_nG := \pi_1(\UConf_n(G))$. Then for each $i$ there is a polynomial $P_i \in \Q[x]$ of degree $\Delta_G^i-1$, such that for all $n \geq 0$,
\[
P_i(n) = \dim_\Q(H_i(B_nG;\Q)).
\]
\text{}\\
\end{thmab}

\begin{remark}
It follows from a Theorem of Ghrist \cite[Theorem 3.3]{Gh} that given any graph $G$ not homeomorphic to $S^1$, $H_i(B_nG) = 0$ for all $i$ strictly greater than the number of essential vertices of $G$. This is realized in the case where $G$ is a tree in Theorem \ref{bettipolyrefine} by the fact that $\Delta_G^i - 1 = -1$ in these cases.\\
\end{remark}

The polynomial $P_i$ of Theorem \ref{bettipolyrefine} is explicitly computed throughout the course of this work (see Theorem \ref{explicitpoly}). This computation implies something somewhat surprising about these homology groups.\\

\begin{corab}\label{degreesequence}
Let $G$ be a tree, and let $i \geq 0$. Then the homology groups $H_i(B_nG)$ depend only on $i, n$, and the degree sequence of $G$.\\
\end{corab}

It is interesting to note that the rank of $H_i(B_nG)$ agrees with a polynomial for all $n \geq 0$, as opposed to only agreeing for $n$ sufficiently large. In the case of configuration spaces of manifolds, a result of this kind does have precedent. We have already discussed the result of Church, Ellenberg, and Farb which states that if $\N$ is an oriented manifold, which is the interior of a manifold with non-empty boundary, then for any $i$ the dimension of $H_i(\Conf_n(\N);\Q)$ agrees with a polynomial for all $n$ \cite[Theorem 6.4.3]{CEF}. It is perhaps an interesting observation that trees can be thought of as graphs with non-trivial boundary. It is unclear whether the connection to the work of Church, Ellenberg, and Farb goes any deeper than this, however. \\

\subsection{An Outline of the proof}

To prove Theorem \ref{bettipolyrefine}, we will rely on classical techniques in commutative algebra, as well as more modern techniques in combinatorial topology. The first key ingredient is the discrete Morse theory of Forman \cite{Fo}. Given a regular CW complex $X$ (see Definition \ref{regcw}), write $\K_i$ for the set of $i$-cells of $X$. A \textbf{discrete Morse function} is a map $f$ from the cells of $X$ to $\R$ satisfying the following two hypotheses for all cells $\sigma \in \K_i$:

\begin{enumerate}
\item $|\{\tau \in \K_{i+1} \mid \sigma \subseteq \overline{\tau} \text{ and } f(\sigma) \geq f(\tau)\}| \leq 1$; \label{dm1}
\item $|\{\tau \in \K_{i-1} \mid \tau \subseteq \overline{\sigma} \text{ and } f(\sigma) \leq f(\tau)\}| \leq 1$. \label{dm2}
\end{enumerate}

We call $\sigma \in \K_i$ a \textbf{critical $i$-cell} with respect to $f$, if the sets of conditions $\ref{dm1}$ and $\ref{dm2}$ are both empty. The main consequence of discrete Morse theory is that the critical cells of $X$ determine its homotopy type. Formally, let $a \in \R$ and write $X(a)$ for the subcomplex of $X$ comprised of the closures of cells $\sigma$ with $f(\sigma) \leq a$. If there is no critical cell $\sigma$ with $a < f(\sigma) < b$, then $X(a)$ is a deformation retract of $X(b)$. Otherwise, if there is a unique critical cell $\sigma$ with $a < f(\sigma) < b$, then $X(b)$ is obtained from $X(a)$ by attaching the cell $\sigma$. Moreover, there is a complex
\begin{eqnarray}
\ldots \rightarrow \M_i \rightarrow \ldots \stackrel{\widetilde{\partial}}\rightarrow \M_1 \rightarrow \M_0 \rightarrow 0\label{mc}
\end{eqnarray}
where $\M_i$ is a free $\Z$-module on the critical $i$-cells of $X$, whose homology is the homology of the space $X$. We call the differential $\widetilde{\partial}$ the \textbf{Morse differential} (see Definition \ref{morsediff}).\\

Using work of Abrams \cite{A}, Farley and Sabalka were able to impose a discrete Morse structure on the spaces $\UConf_n(G)$, where $G$ is any graph \cite{FS}. We will spend a good amount of time recounting the construction of Farley and Sabalka in Section \ref{confg}. Once we have accomplished this, our strategy will be to develop a strong understanding of the critical cells. More specifically, we will work towards understanding the changing behaviors of the critical cells as $n$ varies.\\

Let $G$ be a graph with $E$ essential edges, and write $S_G := \Z[x_1,\ldots,x_E]$ for the integral polynomial ring in $E$-variables. We will prove the following in Section \ref{morsecomplexismodule}.\\

\renewcommand{\labelitemi}{(\textasteriskcentered)}
\begin{itemize}
\item For each $i$, there exists a finitely generated graded $S_G$-module $\M_{i,\dt}$ for which $\M_{i,n}$ is a free $\Z$-module with basis vectors indexed by the critical $i$-cells of $\UConf_n(G)$.\\
\end{itemize}

Specializing to the case where $G$ is a tree, work of Farley \cite{Fa} implies that the Morse differential is always trivial. Using the complex (\ref{mc}) we obtain the following.\\

\begin{thmab}\label{homologyismodule}
Let $G$ be a tree with $E$ essential edges, and let $S_G$ denote the integral polynomial ring in $E$ variables. Then for each $i$ and $n$, the action of $S_G$ on the critical cells of $\UConf_n(G)$ described in Section \ref{morsecomplexismodule} imposes the structure of a finitely generated graded $S_G$-module on the abelian group $\bigoplus_{n \geq 0} H_i(\UConf_n(G))$. In particular, there exists a finitely generated graded $S_G$-module $\mathcal{H}_i$ such that,
\[
(\mathcal{H}_i)_n \cong H_i(\UConf_n(G)).
\]
\text{}\\
\end{thmab}

A result of Ghrist and Abrams (see Theorem \ref{asp}) implies the spaces $\UConf_n(G)$ are aspherical. It follows immediately from this that $H_i(\UConf_n(G)) \cong H_i(B_nG)$. With this in mind, the first part of Theorem \ref{bettipolyrefine} simply follows from the existence of the Hilbert polynomial. Of course, the above theorem does not tell us anything about the degree of the Hilbert polynomial, nor does it bound its obstruction. To accomplish this, we must first prove a structure theorem about the modules $\mathcal{H}_i$.\\

To state this theorem, we first recall the definition of a \textbf{squarefree monomial ideal}. We say an ideal $I \subseteq \Q[x_1,\ldots,x_d]$ is a squarefree monomial ideal, if it contains a generating set of monomials, none of which are divisible by a square in $\Q[x_1,\ldots,x_d]$. These ideals are the subject of Stanley-Reisner theory, and have many desirable properties. For instance, much is known about their Hilbert polynomial (see \cite{MS} for a reference on the subject).\\

\begin{thmab}\label{sqfr}
Let $G$ be a tree, and let $\mathcal{H}_i$ denote the $S_G$-module of Theorem \ref{homologyismodule}. Then $\mathcal{H}_i \otimes_\Z \Q$ is isomorphic to a direct sum of graded twists of squarefree monomial ideals, each having dimension at most $\Delta^i_G$.\\
\end{thmab}

We will find that this theorem implies the conclusions of Theorem \ref{bettipolyrefine}. In fact, we will be able to compute the polynomial $P_i$ associated to $\mathcal{H}_i$ explicitly in terms of invariants of the tree $G$ (see Theorem \ref{explicitpoly}).\\

\subsection{An overview of the paper}

In the next section, we will spend time developing necessary background. This includes short summaries of discrete Morse theory (Section \ref{dmorsethy}), and configuration spaces of graphs (Sections \ref{confg} and \ref{dmorse}). Following this, we use the machinery developed in these preliminary sections to prove the statement (\textasteriskcentered) (Section \ref{morsecomplexismodule}). Finally, we specialize the the case of trees, and use enumerative combinatorial methods to prove Theorem \ref{bettipolyrefine} via an explicit computation of the Hilbert polynomial (Sections \ref{caseoftrees} and \ref{hilbertcomp}).\\

To finish the paper, we briefly consider the case of a general graph $G$. Note that while most of the explicit results in this paper are limited to the case where $G$ is a tree, the statement (\textasteriskcentered) will hold for any graph $G$. A result like Theorem \ref{homologyismodule} will therefore hold for general graphs so long as we know that the Morse differential commutes with the action of $S_G$ on $\M_{i,\dt}$. It is the belief of the author that the action of $S_G$, or perhaps a slight alteration there of, will indeed commute with this differential. Unfortunately, it is known that this differential can become tremendously complicated as $G$ increases in complexity (see \cite{KP}). In any case, the methods in this work therefore provide at least a strategy for proving stability results for more general graphs. In the final sections, we discuss some implications of this.\\

\section*{Acknowledgments}
The author would like to send thanks to Jordan Ellenberg for his many insights, as well as his help in editing this paper. The author would also like to send special thanks to Steven Sam for his vital advice in approaching the primary problem of this work. A great amount of gratitude must also be sent to Daniel Farley, whose expertise in the field was invaluable to the author during his learning of the material. Finally, the author would like the acknowledge the generous support of the National Science Foundation through NSF-RTG grant DMS-502553.\\

\section{Preliminaries}
\subsection{Discrete Morse theory}\label{dmorsethy}

We now take the time to briefly summarize the key points in Forman's discrete Morse theory \cite{Fo}. We will largely be following the exposition of Forman \cite{Fo}\cite{Fo2}, Farley and Sabalka \cite{FS}, and Ko and Park \cite{KP}.\\

In the introduction, we spent some time discussing the notion of discrete Morse function. One thing that should stand out about this definition is that the literal values of the function are immaterial. Namely, the classification of critical cells is unchanged by composition with any strictly monotone function $\R \rightarrow \R$. In many cases it is often easier to construct the relationships between the cells, rather than the discrete Morse function itself. This hints towards the construction of what are known as discrete vector fields. We will use this approach during the exposition of this, and all future sections.\\

To begin, we first must place certain light restrictions on the spaces we will be working with.\\

\begin{definition}\label{regcw}
Let $X$ be a CW complex. A \textbf{cell} of $X$ will always refer to an open cell in $X$. Given a cell $\sigma$ of dimension $i$, we will often write $\sigma^{(i)}$ to indicate that $\sigma$ has dimension $i$. We will write $\K$ to denote the set of cells of $X$, and $\K_i$ to denote the set of $i$-cells of $X$.\\

A cell $\tau^{(i)} \subseteq \overline{\sigma^{i+1}}$ is said to be a \textbf{regular face} of a cell $\sigma^{(i+1)}$ if, given a characteristic map $\Phi_\sigma:D^{i+1} \rightarrow X$ for $\sigma^{(i+1)}$, $\Phi_{\sigma}^{-1}(\tau)$ is a closed ball, and the map $\Phi_{\sigma}|_{\Phi_\sigma^{-1}(\tau)}$ is a homeomorphism. We say that the complex $X$ is \textbf{regular} if given a pair of cells $\tau^{(i)} \subseteq \overline{\sigma^{(i+1)}}$, $\tau^{(i)}$ is a regular face of $\sigma^{(i+1)}$. Equivalently, $X$ is regular if and only if the attaching map of each of its cells is a homeomorphism.
\end{definition}

We will assume throughout most of our exposition that $X$ is a regular CW complex. In his original paper on discrete Morse theory \cite{Fo}, Forman proves some of his results without the requirement that $X$ be regular. The spaces $\UConf_n(G)$, for any graph $G$, are actually cubical complexes, which are certainly regular CW complexes. Therefore, the condition that $X$ be regular is not restrictive for what we need.

\begin{definition}
Let $X$ be a regular CW complex. A \textbf{discrete vector field} $V$ on $X$ is a collection of partially defined functions $V_i:K_i \rightarrow K_{i+1}$ satisfying the following three conditions for each $i$:
\begin{enumerate}
\item $V_i$ is injective;
\item the image of $V_i$ is disjoint from the domain of $V_{i+1}$;
\item for any $\sigma^{(i)}$ in the domain of $V_i$, $\sigma^{(i)}$ is a face of $V_i(\sigma^{(i)})$.\\
\end{enumerate}

Given a regular CW complex $X$ equipped with a discrete vector field $V$, a \textbf{cellular path} between two cells $\alpha^{(i)}$ and $\beta^{(i)}$ is a finite sequence of $i$-cells
\[
\alpha^{(i)} = \alpha_0^{(i)}, \alpha_1^{(i)}, \ldots, \alpha_{l-1}^{(i)}, \alpha_l^{(i)} = \beta^{(i)}
\]
such that $\alpha_{j+1}^{(i)}$ is a face of $V_i(\alpha_{j}^{(i)})$. We say that the path is \textbf{closed} if $\alpha^{(i)} = \beta^{(i)}$, and we say it is \textbf{trivial} if $\alpha_j^{(i)} = \alpha_{k}^{(i)}$ for all $j,k$.\\

A discrete vector field $V$ is said to be a \textbf{discrete gradient vector field} if it admits no non-trivial closed cellular paths.\\

If $V$ is a discrete gradient vector field on a regular CW complex $X$, then we call a cell $\sigma$ of $X$ \textbf{redundant} if $\sigma$ is in the domain of $V_i$ for some $i$, \textbf{collapsible} if it is in the image of $V_i$ for some $i$, and \textbf{critical} otherwise.\\
\end{definition}

\begin{proposition}[\cite{FS} Proposition 2.2, \cite{Fo} Theorem 3.4]\label{VFdecomp}
Let $X$ be a regular CW complex equipped with a discrete gradient vector field $V$. Consider the filtration
\[
\emptyset = X_0'' \subseteq X_0' \subseteq X_1'' \subseteq X_1' \subseteq \ldots \subseteq X_n'' \subseteq X_n' \subseteq \ldots
\]
where $X_i'$ is the $i$-skeleton of $X$ with the redundant cells removed, and $X_i''$ is the $i$-skeleton of $X$ with both the redundant and critical cells removed. Then:
\begin{enumerate}
\item For any $i$, $X'_i$ is obtained from $X''_i$ by attaching $m_i$ $i$-cells to $X_n''$ along their boundaries, where $m_i$ is the number of critical $i$-cells of the discrete gradient vector field $V$.
\item For any $i$, $X_{i+1}''$ deformation retracts onto $X_i'$.\\
\end{enumerate}
\end{proposition}

The above proposition leads to one notable corollary, which we record now.\\

\begin{corollary}[\cite{FS} Proposition 2.3, \cite{Fo} Corollary 3.5]\label{critdecomp}
Let $X$ be a regular CW complex equipped with a discrete gradient vector field $V$. Then $X$ is homotopy equivalent to a CW complex with precisely $m_i$ $i$-cells for each $i$, where $m_i$ is the number of critical $i$-cells of $V$.\\
\end{corollary}

Just as is the case with traditional Morse theory, the decomposition of the space $X$ given by Corollary \ref{critdecomp} can be used to compute the homology groups of $X$. For simplicity, we will state the construction for cubical complexes, although the general case is similar.\\

\begin{definition}\label{morsediff}
Let $X$ be a cubical complex, and write $C_i(X)$ for the free abelian group of $i$-cells of $X$. If $c \cong I_1 \times \ldots \times I_i$ is an $i$-cell of $X$, with $I_j \cong [0,1]$, then we define
\[
c^j_{\tau} = I_1 \times \ldots \times I_{j-1} \times \{0\} \times I_{j+1} \times \ldots \times I_i, \hspace{.5cm} c^j_{\iota} = I_1 \times \ldots \times I_{j-1} \times \{1\} \times I_{j+1} \times \ldots \times I_i.
\]
This allows us to define a boundary morphism
\[
\partial: C_i(X) \rightarrow C_{i-1}(X)
\]
given by
\[
\partial(c) := \sum_j c^j_\tau - c^j_\iota,
\]
turning $C_\dt(X)$ into a chain complex. It is a well known fact that the homology of this chain complex is the usual homology of the space $X$.\\

Further assume that $X$ is equipped with a discrete gradient vector field $V$. Then we have a map $R:C_i(X) \rightarrow C_i(X)$ defined by
\[
R(c) = \begin{cases} 0 &\text{ if $c$ is collapsible}\\ c &\text{ if $c$ is critical}\\ \pm \partial(V_i(c)) + c &\text{ otherwise,}\end{cases}
\]
where the sign of $\partial(V_i(c))$ in the above definition is chosen so that $c$ has a negative coefficient. The property that $V$ has no non-trivial closed paths implies that $R^m(c) = R^{m+1}(c)$ for all $m \gg 0$ and all $i$-cells $c$ \cite{Fo}. We set $R^\infty(c)$ to be this stable value.\\

For each $i$, let $\M_i$ denote the free abelian group with basis indexed by the critical $i$-cells of $V$. Then the \textbf{Morse complex} associated to $V$ is defined to be
\[
\M_\dt : \ldots \rightarrow \M_n \rightarrow \ldots \rightarrow \M_1 \stackrel{\widetilde{\partial}}\rightarrow \M_0 \rightarrow 0,
\]
where boundary map $\widetilde{\partial}$ is given by
\[
\widetilde{\partial}(c) := R^\infty(\partial(c))
\]
The map $\widetilde{\partial}$ is known as the \textbf{Morse differential}.\\
\end{definition}

\begin{theorem}[Theorem 8.2 \cite{Fo}, Theorem 7.3 \cite{Fo2}]\label{morsediffcomp}
For all $i$ there are isomorphisms,
\[
H_i(X) \cong H_i(\M_\dt)
\]
\text{}\\
\end{theorem}

\subsection{The configuration spaces of graphs}\label{confg}

In this section we review necessary facts about configuration spaces of graphs. In the next section, we will explain how the techniques of discrete Morse theory apply to these spaces.\\

\begin{definition}
A \textbf{graph} is any compact, connected CW complex of dimension one. A \textbf{tree} is a topologically contractible graph. Given a graph $G$ with vertex $v$, we will write $\mu(v)$ to denote the degree of the vertex $v$. Note that any loop on $v$ contributes 2 to its degree count. We say that $v$ is \textbf{essential} if $\mu(v) \geq 3$. If $\mu(v) = 1$, we say that $v$ is a \textbf{boundary} vertex, and the unique edge connected to $v$ is called a \textbf{boundary edge}.\\

The \textbf{configuration space on $n$ points} of $G$ is the topological space,
\[
\Conf_n(G) := \{(x_1,\ldots,x_n) \in G^n \mid x_i \neq x_j \text{ if } i \neq j\}.
\]
We note that there is a natural fixed-point-free action on $\Conf_n(G)$ by the symmetric group $\Sn_n$. The \textbf{unordered configuration space on $n$ points} of $G$ is the quotient space,
\[
\UConf_n(G) := \Conf_n(G)/\Sn_n.
\]
\text{}\\
\end{definition}

For the majority of this paper, we will work with the spaces $\UConf_n(G)$. Note that many of the structural theorems discussed in this section will apply to both spaces.\\

In order to apply discrete Morse theory to questions about these configuration spaces, we will first need to place a CW complex structure on them. To accomplish this, we use a theorem of Abrams \cite{A}.\\

\begin{definition}
The \textbf{discretized configuration space on $n$ points} over $G$, $D_n(G)$, is the CW subcomplex of $G^n$ spanned by cells of the form
\[
\sigma_1 \times \ldots \times \sigma_n
\]
where $\sigma_i \subseteq G$ is a cell (i.e. an edge or vertex) of $G$, and $\overline{\sigma_i} \cap \overline{\sigma_j} = \emptyset$ whenever $i \neq j$. We write $UD_n(G)$ to denote the quotient of $D_n(G)$ by the action of the symmetric group.\\
\end{definition}

\begin{theorem}[\cite{A} Theorem 2.1, \cite{KKP} Theorem 2.4, \cite{PS}]\label{cellstructure}
Let $G$ be a graph, and assume that $G$ satisfies the following two properties:
\begin{enumerate}
\item each path connecting distinct vertices of degree $\neq 2$ has length at least $n-1$;
\item each homotopically essential path connecting a vertex to itself has length at least $n+1$.
\end{enumerate}
Then the inclusions $D_n(G) \hookrightarrow \Conf_n(G)$, and $UD_n(G) \hookrightarrow \UConf_n(G)$ are homotopy equivalences.\\
\end{theorem}

\begin{remark}
Note that in the first cited source, Abrams states the theorem assuming that each path connecting distinct vertices of degree $\neq 2$ has length at least $n+1$. It is noted after the proof that the version of the theorem stated above is true, and a brief argument is given for how it is proven. In the second source, Kim, Ko, and Park give a formal argument for this improvement. In the third source, Prue and Scrimshaw provide a discrete Morse theory argument, which is independent of the first two sources. For our purposes, the exact number of vertices needed is unimportant, as we can always just subdivide the edges of $G$ more if needed.\\
\end{remark}

Note that subdividing edges of a graph $G$ does not impact the configuration spaces $\Conf_n(G)$ and $\UConf_n(G)$. For much of what follows, we will often just assume, without explicit mention, $G$ is subdivided enough so that the homotopy equivalence of Theorem \ref{cellstructure} holds.\\

Theorem \ref{cellstructure} implies that $\Conf_n(G)$ and $\UConf_n(G)$ are homotopy equivalent to cubical complexes of dimension $n$. In fact, we will be able to do better than this.\\

\begin{theorem}[\cite{Gh} Theorem 3.3]\label{dimconf}
Let $G$ be a graph which is not homeomorphic to $S^1$. Then $\Conf_n(G)$ and $\UConf_n(G)$ are homotopy equivalent to CW complexes of dimension $N_G$, where $N_G$ is the number of essential vertices of $G$.\\
\end{theorem}

\begin{remark}
If $G$ is homeomorphic to $S^1$, then $\UConf_n(G)$ is easily seen to be homotopy equivalent to a circle for all $n$. Throughout this work, and the literature in general, it is a recurring theme that certain theorems only apply to graphs which are neither $S^1$ nor the interval. It is interesting to observe that these two graphs are precisely those which are homeomorphic to compact manifolds.\\
\end{remark}

Note that Ghrist originally proved Theorem \ref{dimconf} using more classical topological means. We will later see it naturally falls out of the discrete Morse structure that Farley and Sabalka placed on $UD_n(G)$. This was first noted by Farley and Sabalka in \cite{FS}.\\

One remarkable thing to note about Ghrist's theorem is that the dimension of the configuration spaces of graphs is independent of $n$. This behavior is in stark contrast to the behavior of configuration spaces of smooth manifolds of dimension $\geq 2$.\\

In his paper \cite{Gh}, Ghrist also proves that configuration spaces of graphs are, in fact, aspherical. This result was later reproven by Abrams \cite[Theorem 3.10]{A}, where he shows that both $D_n(G)$ and $UD_n(G)$ are universally covered by CAT(0) complexes.\\

\begin{theorem}[\cite{Gh} Theorem 3.1, \cite{A} Corollary 3.11]\label{asp}
Let $G$ be a graph. Then $\Conf_n(G)$, and hence $\UConf_n(G)$, are aspherical. That is to say, $\pi_k(\Conf_n(G)) = 0$ for $k > 1$.\\
\end{theorem}

Note that this theorem is analogous to that which says the configuration spaces of the plane are aspherical. In that case, the fundamental groups of the ordered and unordered spaces are the Artin pure braid groups, and the Artin braid groups, respectively. We borrow this terminology for our context as well.\\

\begin{definition}
Let $G$ be a graph. The \textbf{braid group on $n$ strands} of $G$ is defined to be
\[
B_nG := \pi_1(\UConf_n(G)).
\]
We similarly define the \textbf{pure braid group on $n$ strands} as
\[
P_nG := \pi_1(\Conf_n(G))
\]
\text{}\\
\end{definition}

The study of the braid groups of graphs is still a very active area of research. See \cite{KKP}\cite{KP}\cite{FS2} for more on these groups.\\

As an immediate corollary to Theorem \ref{asp}, we obtain the following.\\

\begin{corollary}
Let $G$ be a graph. then there are isomorphisms
\[
H_i(\Conf_n(G)) \cong H_i(P_nG), H_i(\UConf_n(G)) \cong H_i(B_nG).
\]
\text{}\\
\end{corollary}

The goal for this paper is to establish a methodology for understanding stability phenomena in the groups $H_i(B_nG)$, in the spirit of modern trends of asymptotic algebra. Note that we will spend very little time considering pure braid groups. In fact, there doesn't seem to be much literature about the homology of these groups, as they are vastly more complicated than the braid groups \cite{KP}\cite{BF}.\\

The following theorem of Farley exactly computes the homology groups when $G$ is a tree.\\

\begin{theorem}[\cite{Fa}]\label{treefree}
Let $G$ be a tree. Then the group $H_i(B_nG)$ is free for all $n \geq 0$.
\end{theorem}

Note that the theorem of \cite{Fa} is slightly more general than this. We will discuss the more general version in later sections.\\

Expanding upon the work of Farley, the following theorem of Ko and Park suggests that the tree case is indicative of a more general phenomenon.\\

\begin{theorem}[\cite{KP} Theorem 3.6]\label{kkpthm}
Let $G$ be a graph. Then $G$ is planar if and only if $H_1(B_nG)$ is torsion free for any, and therefore all, $n\geq 2$. Moreover, in the case where $G$ is not planar, all torsion is 2-torsion.\\
\end{theorem}

Kim, Ko, and Park proved that $G$ is planar if and only if $H_1(B_2G,\Z)$ is torsion free in \cite[Theorem 5.5]{KKP}. It is also conjectured in that work that Theorem \ref{kkpthm} is true. It should be noted that \cite[Theorem 3.6]{KP} is far stronger than what we have written above. In fact, their result explicitly computes the groups $H_1(B_nG,\Z)$ in terms of combinatorial invariants of the graph $G$. It follows from their computation that the amount 2-torsion of the group is unvarying in $n$. Moreover, if $G$ is biconnected - that is, $G$ requires the removal of at least 2 vertices to disconnect it - then $H_1(B_nG,\Z) \cong H_1(B_2G,\Z)$ for all $n \geq 2$ \cite[Lemma 3.12]{KP}. In the final section of this paper we will conjecture an extension of this fact to the higher homologies (see Conjecture \ref{mainconj}).\\

To finish this section, we record a result of Gal on the Euler characteristic of these spaces. Note that Gal proves a more general theorem for computing the Euler characteristic of the configuration space of any simplicial complex.\\

\begin{theorem}[\cite{Ga} Theorem 2]\label{eulercharacteristic}
Let $G$ be a graph, and set
\[
\mathfrak{e}(t) := \sum_{n \geq 0} \frac{\chi(\Conf_n(G))}{n!}t^n.
\]
Then,
\[
\mathfrak{e}(t) = \frac{\prod_{v} (1-(1-\mu(v))t)}{(1-t)^E}
\]
where $E$ is the number of edges of $G$.\\
\end{theorem}

\begin{remark}
$\Conf_n(G)$ is an $n!$-fold cover of $\UConf_n(G)$. It follows from this that the above formula can be easily used to compute the Euler characteristic of $\UConf_n(G)$ as well.\\
\end{remark}

\begin{definition}
Let $G$ be a graph which has at least one essential vertex, and let $\mathcal{E}$ denote the set of essential vertices of $G$. Then an \textbf{essential edge} of $G$ is a connected component of $G - \mathcal{E}$.\\
\end{definition}

\begin{corollary}\label{eulerpolynomial}
Let $G$ be a graph with at least one essential vertex, and let $E$ denote the number of essential edges of $G$. Then there is a polynomial $P$ of degree $E-1$, for all $n \geq 0$, such that the function
\[
n \mapsto \chi(\UConf_n(G))
\]
is equal $P(n)$.
\end{corollary}

\begin{proof}
We first note that smoothing the degree 2 vertices of $G$ does not impact the spaces $\UConf_n(G)$. We may therefore assume without loss of generality that $G$ does not have any such vertices. In this case, the Euler characteristic $\chi(\UConf_n(G))$ is the coefficient of $t^n$ in the power series expansion of
\[
\frac{\prod_{v \text{ essential}} (1-(1-\mu(v)t))}{(1-t)^{E}}
\]
A straight forward enumerative combinatorics argument implies that the $n$-th coefficient of the power series expansion of this expression is a polynomial of degree exactly $E-1$ for $n$ sufficiently large. It remains to show that this agreement begins at $n = 0$.\\

Assume that $G$ has $N_G$ essential vertices. It is a standard fact from enumerative combinatorics that the coefficients of the power series expansion of a rational function of the form $\frac{f(x)}{(1-t)^l}$ will agree with a polynomial for all $n$ so long as $\deg(f) < l$. This is the case in our specific instance, so long as $N_G > 1$.
\end{proof}

Theorem \ref{eulercharacteristic} was largely the inspiration for this work. It suggests that the asymptotically the Betti numbers of $H_i(B_nG)$ are polynomial. In this work we will prove this suggestion in the case where $G$ is a tree. We will also provide a setup in the case where $G$ is a general graph, which hopefully will be able to illustrate this behavior in that case as well.\\

\subsection{The discrete Morse theory of $\UConf_n(G)$} \label{dmorse}

In this section we will outline the discrete Morse structure on $\UConf_n(G)$ developed in the work of Farley and Sabalka \cite{FS}. We begin by fixing $n$, and assuming the graph $G$ is sufficiently subdivided for Theorem \ref{cellstructure} to hold for $\UConf_n(G)$.\\

Fix a spanning tree $T$ for $G$, as well as an embedding of $T$ into the plane. We will label the vertices of $T$ by applying a depth-first search. Concretely, begin by choosing a boundary vertex $v_0$ to be the root of $T$, and label it with the number 0. Continue down the boundary edge adjacent to $v_0$, labeling any vertices encountered along the way with increasing labels. If at any point an essential vertex is encountered, then travel down the leftmost - relative to the current direction of travel - edge whose vertices have not been labeled. If a boundary vertex is encountered, then one returns to the most recently passed essential vertex. An example of a correctly labeled tree is given in Figure \ref{labeledtree}.\\

\begin{figure}
\begin{tikzpicture}
    \tikzstyle{every node}=[draw,circle,fill=white,minimum size=4pt,
                            inner sep=0pt]
    \draw (0,0) node (0) [label=above:$0$] {}
		      --(.5,0) node (1) [label=above:$1$]{}
					--(1,0) node (2) [label=right:$2$]{}
					--(1,.5) node (3) [label=left:$3$]{}
					--(1,1) node (4) [label=left:$4$]{}
					--(1,1.5) node (5) [label=left:$5$]{}
					--(1,2) node (6) [label=left:$6$]{}
    			(4) -- (1.5,1) node (7) [label=above:$7$]{}
					--(2,1) node (8) [label=above:$8$]{}
					(2) -- (1,-.5) node (9) [label=left:$9$]{}
					--(1,-1) node (10) [label=left: $10$]{};

\end{tikzpicture}
\caption{A tree which is properly labeled. Note that this tree is also sufficiently subdivided to apply Theorem \ref{cellstructure} with $n = 3$.}\label{labeledtree}
\end{figure}
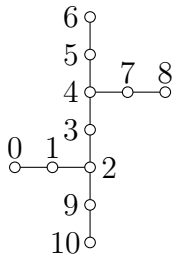

For the remainder of this section, $G$ and $T$ will be as in the previous paragraphs. As before we will write $\K$ to denote the set of cells of $\UConf_n(G)$, while $\K_i$ will denote the set of $i$-cells of $\UConf_n(G)$.\\

\begin{definition}\label{dgvf}
Given an edge $e$ of $G$, we write $\iota(e)$ to denote its largest, with respect to the given labeling, vertex, and $\tau(e)$ to denote its smallest vertex. If $v$ is a vertex of $G$, which is not $v_0$, then we write $e(v)$ to denote the unique edge of $T$ for which $\iota(e(v)) = v$. We note that $e(v)$ will be the first edge on the unique path within $T$ from $v$ to $v_0$.\\

Let $c = \{v_{j_1},\ldots, v_{j_{n-i}}, e_{r_1},\ldots, e_{r_i}\}$ be an $i$-cell of $UD_n(G)$. If
\[
\tau(e(v_j)) \cap \{v_{j_1},\ldots, v_{j_{n-i}}, \iota(e_{r_1}),\tau(e_{r_1}),\ldots, \iota(e_{r_i}),\tau(e_{r_i})\} \neq \emptyset,
\]
then we say that $v_j$ is \textbf{blocked} in $c$. Note that by convention $v_0$ is always blocked.\\

We define a partial function $V_i: \K_i \rightarrow \K_{i+1}$ inductively in the following way. Given an $i$-cell $c = \{v_{j_1},\ldots, v_{j_{n-i}}, e_{r_1},\ldots, e_{r_i}\}$, assume without loss of generality that $v_{j_1}$ is an unblocked vertex of $c$ which is smallest with respect to the labeling of $T$. Assuming that $c$ is not in the image of $V_{i-1}$, we set
\[
V_i(\{v_{j_1},\ldots, v_{j_{n-i}}, e_{r_1},\ldots, e_{r_i}\}) := \{v_{j_2},\ldots, v_{j_{n-i}}, e_{r_1},\ldots, e_{r_i}, e(v_{j_1}))\}
\]
If $c$ does not have any unblocked vertices, or if $c$ is in the image of $V_{i-1}$, then $V_i(c)$ is undefined.\\
\end{definition}

\begin{theorem}[\cite{FS} Sections 3.1 and 3.2]
The collection of partial functions $V_i:\K_i \rightarrow \K_{i+1}$ form a discrete gradient vector field on $UD_n(G)$.\\
\end{theorem}

While the spaces $\UConf_n(G)$ can be high dimensional, we can still visualize their cells. Indeed, it is often useful to think of the cells of $\UConf_n(G)$ as being subsets of edges and vertices of the graph $G$ itself. In Figure \ref{cellexample} we see examples of a critical, a collapsible, and a redundant 1-cell of V, where $G$ is taken to be the tree in Figure \ref{labeledtree} and $n = 3$. In these examples, the cell $\{v_0,v_1,e_{9,10}\}$ is collapsible, as
\[
\{v_0,v_1,e_{9,10}\} = V_0(\{v_0,v_1,v_{10}\}),
\]
and $v_{10}$ is the smallest unblocked vertex. The cell $\{v_1,v_5,e_{4,7}\}$ is redundant, as $v_1$ is not blocked, and it is not in the image of $V_0$. Indeed, the only cell that could possibly map to it would be $\{v_1,v_5,v_7\}$. However, in this case $v_1$ is the smallest unblocked vertex. Finally, the cell $\{v_0,v_5,e_{4,7}\}$ is critical, as its every vertex is blocked, and it is not collapsible.\\

\begin{figure}
\begin{tikzpicture}
\tikzstyle{every node}=[draw,circle,fill=white,minimum size=4pt,
                            inner sep=0pt]
    \draw (0,0) node (0) [fill = black, label=above:$0$] {}
		      --(.5,0) node (1) [label=above:$1$]{}
					--(1,0) node (2) [label=right:$2$]{}
					--(1,.5) node (3) [label=left:$3$]{}
					--(1,1) node (4) [fill = black, label=left:$4$]{}
					--(1,1.5) node (5) [fill = black, label=left:$5$]{}
					--(1,2) node (6) [label=left:$6$]{}
    			(4) -- (1.5,1) node (7) [fill = black, label=above:$7$]{}
					--(2,1) node (8) [label=above:$8$]{}
					(2) -- (1,-.5) node (9) [label=left:$9$]{}
					--(1,-1) node (10) [label=left: $10$]{};
		 \draw [line width=1mm] (4) -- (7);
\end{tikzpicture}
\begin{tikzpicture}
\tikzstyle{every node}=[draw,circle,fill=white,minimum size=4pt,
                            inner sep=0pt]
    \draw (0,0) node (0) [fill = black, label=above:$0$] {}
		      --(.5,0) node (1) [fill = black, label=above:$1$]{}
					--(1,0) node (2) [label=right:$2$]{}
					--(1,.5) node (3) [label=left:$3$]{}
					--(1,1) node (4) [label=left:$4$]{}
					--(1,1.5) node (5) [label=left:$5$]{}
					--(1,2) node (6) [label=left:$6$]{}
    			(4) -- (1.5,1) node (7) [label=above:$7$]{}
					--(2,1) node (8) [label=above:$8$]{}
					(2) -- (1,-.5) node (9) [fill = black, label=left:$9$]{}
					--(1,-1) node (10) [fill = black, label=left: $10$]{};
		 \draw [line width=1mm] (9) -- (10);
\end{tikzpicture}
\begin{tikzpicture}
\tikzstyle{every node}=[draw,circle,fill=white,minimum size=4pt,
                            inner sep=0pt]
    \draw (0,0) node (0) [label=above:$0$] {}
		      --(.5,0) node (1) [fill = black, label=above:$1$]{}
					--(1,0) node (2) [label=right:$2$]{}
					--(1,.5) node (3) [label=left:$3$]{}
					--(1,1) node (4) [fill = black, label=left:$4$]{}
					--(1,1.5) node (5) [fill = black, label=left:$5$]{}
					--(1,2) node (6) [label=left:$6$]{}
    			(4) -- (1.5,1) node (7) [fill = black, label=above:$7$]{}
					--(2,1) node (8) [label=above:$8$]{}
					(2) -- (1,-.5) node (9) [label=left:$9$]{}
					--(1,-1) node (10) [label=left: $10$]{};
		 \draw [line width=1mm] (4) -- (7);
\end{tikzpicture}
\caption{The critical 1-cell $\{v_0,v_5,e_{4,7}\}$, the collapsible 1-cell $\{v_0,v_1,e_{9,10}\}$, and the redundant 1-cell $\{v_1,v_5,e_{4,7}\}$, respectively.}\label{cellexample}
\end{figure}
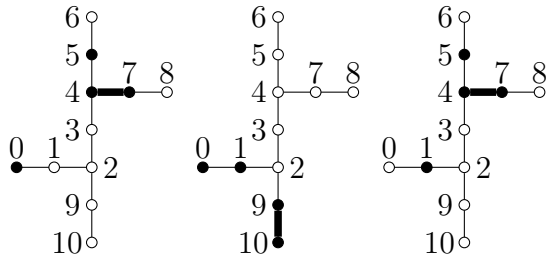

From these examples, one can imagine there existing a way to classify all cells of the three types. To write down such a classification, we first need a bit of nomenclature.\\

\begin{definition}
Let $c = \{v_{j_1},\ldots, v_{j_{n-i}}, e_{r_1},\ldots, e_{r_i}\}$ be an $i$-cell of $UD_n(G)$. We say that an edge $e \in c$ is \textbf{order respecting in $c$} if $e$ is in $T$, and for all $v \in c$ such that $\tau(e(v)) = \tau(e)$, the label of $v$ in $T$ is larger than the label of $\iota(e)$.\\
\end{definition}

\begin{remark}
Note that edges $e \in G-T$ - which we call \textbf{deleted edges} -  are, by definition, never order respecting.\\
\end{remark}

\begin{theorem}[The Classification Theorem, \cite{FS} Theorem 3.6]
Let $c$ be a cell of $UD_n(G)$, and let $T$ be a choice of spanning tree of $G$ equipped with a planar embedding. Then:
\begin{enumerate}
\item $c$ is critical if and only if it contains no order respecting edges and all of its vertices are blocked;
\item $c$ is redundant if and only if
\begin{enumerate}
\item it contains no order respecting edges and at least one of its vertices is unblocked or,
\item it contains an order respecting edge (and thus a minimal order respecting edge e) and there is some unblocked vertex $v$ such that the label of $v$ in $T$ is less than that of $\iota(e)$;
\end{enumerate}
\item $c$ is collapsible if and only if it contains an order respecting edge (and thus a minimal order respecting edge $e$) and, for any vertex $v$ such that the label of $v$ in $T$ is less than that of $\iota(e)$, $v$ is blocked.\\
\end{enumerate}
\end{theorem}

It is not hard to check that the previous examples satisfy the conditions of the Classification Theorem.\\

One useful fact about the discrete gradient vector field $V$ is that its critical cells are, in some sense, the most restrictive of the three possible types. Indeed, we may expand upon the classification of critical cells in the following way.\\

\begin{lemma}\label{critlemma}
Let $c = \{v_{j_1},\ldots, v_{j_{n-i}}, e_{r_1},\ldots, e_{r_i}\}$ be a critical $i$-cell of $V$.
\begin{enumerate}
\item For each $k$ in $\{1,\ldots,i\}$, either $e_{r_k}$ is a deleted edge, or $\tau(e_{r_k})$ is an essential vertex of $T$.
\item Let $c'$ be a critical $i$-cell with edges $e_{r_1},\ldots,e_{r_i}$, such that the number of its vertices in each component of $T - \{\overline{e_{r_1}},\ldots,\overline{e_{r_i}}\}$ agrees with the number of vertices of $c$ in these same components. Then, $c = c'$.
\end{enumerate}
\end{lemma}

\begin{proof}
For the first statement, let $e$ be an edge of $c$ which is not deleted. If $\tau(e) = v_0$, then it is clear that $e$ is order respecting. We note that $\tau(e)$ is not a boundary vertex in any other case by how the labeling on $T$ was defined. It therefore remains to show that $\tau(e)$ does not have degree 2. Indeed, if this were the case, then for any vertex $v$ of $G$, $\tau(e(v)) \neq \tau(e)$. In particular, $e$ must be order preserving. This proves the first claim.\\

For the second claim, one uses the fact that the vertices of $c$ must all be blocked.\\
\end{proof}

Note that the observations made in the above lemma were originally pointed out by Farley and Sabalka when they defined the vector field $V$ \cite[Section 3.3]{FS}. We collected these observations in Lemma \ref{critlemma} so that they could be easily referred back to during future sections. We will find that this lemma is critical in defining the polynomial ring structure on $\bigoplus_n H_i(B_nG)$. In the next section, we will find that the deleted edges of $G$ can be chosen so that they remain unchanged as $n$ increases and we repeatedly subdivide the graph. The above lemma also suggests that those edges of critical cells contained in the tree $T$ are unchanging in $n$. While this is not literally the case, as all edges of $T$ are being subdivided as $n$ increases, Lemma \ref{critlemma} tells us that the important information encoded by such an edge is the essential vertex of its tail, as well as the direction it is leaving this essential vertex. For this reason we will often be a bit loose with our language and claim that two critical cells with different choices of $n$ have the ``same'' edges.\\

\begin{definition}
Let $c$ be a critical cell of the discrete gradient vector field, and let $e \in T$ be an edge in $c$. The classification theorem implies that there must exist some vertex $v \in c$ such that the label of $v$ in $T$ is smaller than that of $\iota(e)$, while $\tau(e(v)) = \tau(e)$. We refer to such a vertex as a \textbf{witness} for $e$ in $c$. If $v$ is the only witness for $e$ in $c$, then we say $v$ is \textbf{necessary}.\\
\end{definition}

\section{The proof Theorem \ref{bettipolyrefine}}

\subsection{The setup}

We use this section to fix the notation which will be used throughout the proofs of the main theorems. Note that all of the constructions presented in the previous sections were already well established in the literature prior to this work. To the knowledge of the author, the main construction of the following sections - namely the action of $S_G$ on the critical cells of $\UConf_n(G)$ - does not appear elsewhere in the literature.\\

As was the case in all previous sections, let $G$ denote a graph, which is neither a circle nor a line segment. We will reserve $E$ to denote the number of essential edges of $G$, while $N_G$ will denote the number of essential vertices of $G$.\\

Next, we must construct our spanning tree $T$, as in Section \ref{dmorse}. To do this, we first subdivide every edge of $G$, which connects two essential vertices, once. Note that this includes any loops of $G$. Once this is done, we choose a spanning tree $T$ of $G$ which satisfies the following:

\begin{eqnarray}
\text{Every edge of $G$ not in $T$ is adjacent to an essential vertex of $G$.}\label{treecondition}
\end{eqnarray}

Note that it is not entirely obvious that such a $T$ exists for an arbitrary graph $G$. This fact was proven by Farley and Sabalka during the proof of \cite[Theorem 4.4]{FS}.\\

\begin{remark}
It is noted by Farley and Sabalka that if one chooses $T$ to satisfy (\ref{treecondition}), then Lemma \ref{critlemma} and Corollary \ref{critdecomp} imply Theorem \ref{dimconf}.\\
\end{remark}

Having chosen our spanning tree $T$ to satisfy (\ref{treecondition}), we observe that for $n \geq 2$, we can sufficiently subdivide $G$ for Theorem \ref{cellstructure} by subdividing $T$. This follows from the fact that we began by subdividing all edges of $G$ which connected essential vertices. We will often not differentiate the spanning trees chosen for each $n$ for this very reason.\\

\subsection{The modules $\M_{i,\dt}$}\label{morsecomplexismodule}

Recall from Section \ref{dmorsethy} the Morse complex
\[
\ldots \rightarrow \M_i \rightarrow \ldots \rightarrow \M_1 \rightarrow \M_0 \rightarrow 0
\]
where $\M_i$ is the free $\Z$-module with basis indexed by critical cells. In our case, we write $\M_{i,n}$ to denote the free $\Z$-module with basis indexed by the critical cells of the Farley-Sabalka discrete gradient vector field of $UD_n(G)$.\\

Let $\mathfrak{e}_1,\ldots,\mathfrak{e}_E$ denote the essential edges of the graph $G$. Then we set
\[
S_G := \Z[x_{\mathfrak{e}_1},\ldots,x_{\mathfrak{e}_{E}}]
\]

Our goal for the remainder of this section will be to argue, for each $i$, that $\M_{i,\dt}$ has the structure of a finitely generated graded $S_G$-module.\\

\begin{definition}
Fix $i$ and $n$, and let $\mathfrak{e}$ denote an essential edge of $G$. Given a critical $i$-cell $c = \{v_{j_1},\ldots, v_{j_{n-i}}, e_{r_1},\ldots, e_{r_i}\}$ of $UD_n(G)$, we define $x_{\mathfrak{e}} \cdot c$ to be the unique critical $i$-cell of $UD_n(G)$ obtained from $c$ by adding a vertex to the connected component of $T - \{\overline{e_{r_1}},\ldots,\overline{e_{r_i}}\}$ containing $\mathfrak{e}$. More precisely, $x_{\mathfrak{e}}\cdot c$ is obtained from $c$ by adding the smallest vertex on the connected component containing $\mathfrak{e}$. This is well defined by Lemma \ref{critlemma}, as well as the choice of the tree $T$, as $T$ will contain a unique representative of each essential edge of $G$ by construction. This turns the collection $\{\M_{i,n}\}_n$ into a graded $S_G$-module, which we denote $\M_{i,\dt}$.\\
\end{definition}

One can visualize the action of $S_G$ described in the above definition in the following way. The labeling on $T$ induces a natural flow on $T$. Namely, all edges flow towards the root of $T$. Let $c = \{v_{j_1},\ldots, v_{j_{n-i}}, e_{r_1},\ldots, e_{r_i}\}$ be a critical cell. We imagine the vertices $v_j$ as being particles drifting in the direction of the flow, while the edges $e_j$, along with their endpoints, are stationary blockades. For any essential edge $\mathfrak{e}$, the action of $x_{\mathfrak{e}}$ on $c$ involves placing a new particle somewhere on $\mathfrak{e}$, and allowing it to flow until it too is blocked.\\

We provide an illustration of the action of $S_G$ in Figure \ref{multiplication}.\\

\begin{figure}
\begin{tikzpicture}
\tikzstyle{every node}=[draw,circle,fill=white,minimum size=4pt,
                            inner sep=0pt]
    \draw (0,0) node (0) [fill = black, label=above:$0$] {}
		      --(.5,0) node (1) [label=above:$1$]{}
					--(1,0) node (2) [label=right:$2$]{}
					--(1,.5) node (3) [label=left:$3$]{}
					--(1,1) node (4) [fill = black, label=left:$4$]{}
					--(1,1.5) node (5) [fill = black, label=left:$5$]{}
					--(1,2) node (6) [label=left:$6$]{}
    			(4) -- (1.5,1) node (7) [fill = black, label=above:$7$]{}
					--(2,1) node (8) [label=above:$8$]{}
					(2) -- (1,-.5) node (9) [label=left:$9$]{}
					--(1,-1) node (10) [label=left: $10$]{};
					
	 \draw [line width=1mm] 
				 (4) -- (7);
				
	 \draw[->, line width=.5mm]
		     (3, .5) -- (5,.5);
				
	 \draw (6,0) node (0) [fill = black, label=above:$0$] {}
		      --(6.5,0) node (1) [label=above:$1$]{}
					--(7,0) node (2) [label=above:$2$]{}
					--(7.5,0) node(3)[label = right:$3$]{}
					--(7.5,.5) node (4) [label=left:$4$]{}
					--(7.5,1) node (5) [label=left:$5$]{}
					--(7.5,1.5) node (6) [fill = black, label=left:$6$]{}
					--(7.5,2) node (7) [fill = black,label = left:$7$]{}
					--(7.5,2.5) node (8) [fill = black,label = left:$8$]{}
					--(7.5,3) node (9) [label = left:$9$]{}
					(6) -- (8,1.5) node (10) [fill = black, label = above:$10$]{}
					--(8.5,1.5) node (11) [label = above: $11$]{}
					--(9,1.5) node (12) [label = above: $12$]{}
					(3) -- (7.5,-.5) node (13) [label = left:$13$]{}
					-- (7.5,-1) node (14) [label = left:$14$]{}
					-- (7.5,-1.5) node (15) [label = left:$15$]{};
		\draw [line width = 1mm]
		      (6)--(10);
		
\end{tikzpicture}
\caption{An illustration of multiplication by $x_{\mathfrak{e}}$, where $\mathfrak{e}$ is the essential edge containing the boundary vertex $v_6$ of the tree $G$ of Figure \ref{labeledtree}.}\label{multiplication}
\end{figure}
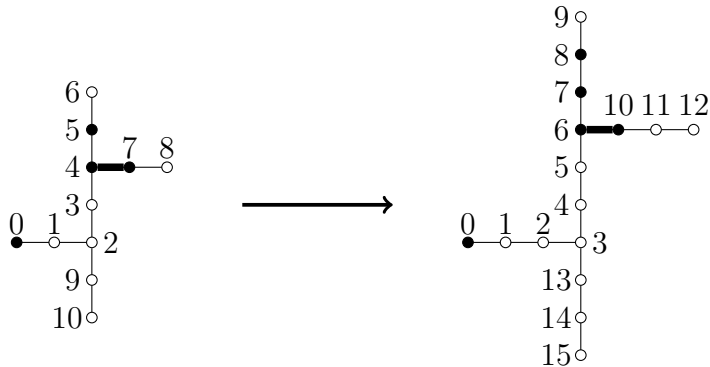

\begin{lemma}\label{fgcomplex}
For each $i$, the module $\M_{i,\dt}$ is finitely generated over $S_G$.\\
\end{lemma}

\begin{proof}
We claim that every critical cell in $\M_{i,n}$, with $n > 2i$, can be obtained from a critical cell in $\M_{i,n-1}$. Indeed, if $c$ is a critical cell in $\M_{i,n}$, then at least one vertex is not a necessary witness of some edge in $e$. Remove the vertex, among those which are not necessary witness vertices, which occupies the maximal position with respect to the labeling on $T$. Removing this vertex leaves us with a critical cell $c'$ in $\M_{i,n-1}$, as this removal cannot create unblocked vertices nor order respecting edges. It follows that $c$ is the image of $c'$ under the action of the appropriate essential edge.\\
\end{proof}

\begin{remark}\label{slowgrowth}
It follows from the above lemma that the number of critical $i$-cells grows, as a function of $n$, like a polynomial of degree $\leq E-1$. In fact, if $i < N_G$, we claim that this polynomial must have degree strictly less than $E-1$. First, we note that the action of $S_G$ on a critical cell $c$ does not affect the edges of $c$. Because of this it follows that $\M_{i,\dt}$ can be expressed as a direct sum of graded $S_G$-modules, where each summand corresponds to a choice of edges in our critical $i$-cell. For any fixed critical $i$-cell $c$, one observes that two variables $x_{\mathfrak{e}}$ and $x_{\mathfrak{e}'}$ will act identically whenever $\mathfrak{e}$ and $\mathfrak{e'}$ are on the same connected component of $G$ with the edges of $c$ removed. It follows that the summand of $\M_{i,\dt}$ containing this cell has Hilbert polynomial of degree strictly less than $E$ whenever $i < N_G$. The Hilbert polynomial of the whole of $\M_{i,\dt}$ is a sum of such polynomials, and therefore also has degree strictly less than $E$. This will be used in the final section of the paper.\\
\end{remark}

Lemma \ref{fgcomplex}, and Theorem \ref{morsediffcomp} now imply some asymptotic data about the homology groups $H_i(B_nG)$.\\

\begin{theorem}
Let $G$ be a graph with $E$ essential edges. Then for all $i \geq 0$, there exists a polynomial $P_i \in \Q[t]$, such that for all $n \gg i$
\[
\dim_\Q(H_i(B_nG;\Q)) \leq P_i(n).
\]
\text{}\\
\end{theorem}

\subsection{The case of trees}\label{caseoftrees}

In this section, we begin to explore the specific case where $G$ is a tree. The work of Farley \cite{Fa} will allow us to conclude quite a bit more than we are able to in the general case. To begin with, we state a refined version of Theorem \ref{treefree}.\\

\begin{theorem}[\cite{Fa}]\label{treefree2}
Let $G$ be a tree. Then for each $n$ the Morse complex $\M_{i,n}$ has trivial differential. In particular,
\[
\mathcal{H}_i \cong \M_{i,\dt},\\
\]
as $S_G$-modules, where $\mathcal{H}_i$ is the module of Theorem \ref{homologyismodule}.\\
\end{theorem}

To state our main stability theorem, we begin with the following definition.\\

\begin{definition}\label{deltaig}
Let $G$ be a tree. Then we define the quantity
\[
\Delta_G^i := \max_{\{\{v_{j_1},\ldots,v_{j_i}\} \mid v_{j_k} \text{ essential}\}} \{ \dim_\Q(H_0(G - \{v_{j_1},\ldots,v_{j_i}\};\Q))\}.
\]
For example, $\Delta_G^1$ is the maximum degree of $G$, while $\Delta_{G}^{N_G}$ is $E$, the number of essential edges. By convention, $\Delta_G^0 = 1$ while $\Delta_G^i = 0$ for $i > N_G$.\\
\end{definition}

\begin{theorem}\label{treedeg}
Let $G$ be a tree, and fix $i\geq 0$. Then the Betti number $b_i(n) := \dim_\Q(H_i(B_n(G);\Q))$ is equal to a polynomial of degree $\Delta_G^i - 1$ for $n \gg i$.\\
\end{theorem}

\begin{proof}
The Theorem \ref{treefree2} and the work of the previous sections imply that $b_i(n)$ is eventually a polynomial; it remains only to show that this polynomial has the claimed degree.\\

We begin by partitioning the basis vectors of $\M_{i,n}$ according to the collection of edges which appear. The action of $S_G$ does not impact the edges, and so each of these partitions will correspond to a summand of $\M_{i,n}$. Moreover, considering any one of these summands, we observe that the variables induce at most $\Delta_G^i$ distinct operators on $\M_{i,n}$. It follows that the degree of the Hilbert polynomial of $\M_{i,\dt}$ is at most $\Delta_G^i-1$. To finish the proof, we will show that $\M_{i,\dt}$ has a summand whose Hilbert polynomial has degree $\Delta_G^i-1$.\\

Fix the $i$ essential vertices of $G$ whose absence realizes the maximum in Definition \ref{deltaig}. For each essential vertex in this fixed list, let $e$ be an edge adjacent to it in the second leftmost direction, relative to the direction of the root. Note that any critical cell containing these edges has a unique essential edge which can house a witness vertex. There is therefore a critical $i$-cell in $\M_{i,2i}$ constructed by choosing each of these edges, along with the unique witness vertex for each. Denote this cell by $c$. The submodule generated by $c$ is a summand of $\M_{i,\dt}$. We claim that this submodule our desired summand.\\

Any $i$-cell of $\M_{i,n}$ which shares the same edges of $c$ - i.e. those basis vectors in the $n$-th graded piece of $S_G\cdot c$ - is entirely determined by distributing $n-2i$ vertices to one of $\Delta_G^i$ components. Thus,
\[
\rank_\Z((S_G\cdot c)_n) = \binom{n-2i + \Delta_G^i - 1}{\Delta_G^i-1}.
\]
as desired.\\
\end{proof}

The only piece of Theorem \ref{bettipolyrefine} which remains to be proven is in showing that $b_i(n)$ agrees with a polynomial for all $n \geq 0$. To accomplish this, we must first prove Theorem \ref{sqfr}. As a first simplification, we note that our answer is unaffected by changing basis to $\Q$, and may therefore assume we are working over a rational polynomial ring. This grants us access to many classical approaches to solving the problem. We will also make heavy use of the direct sum decomposition of $\M_{i,\dt}$ described in the proof of Theorem \ref{treedeg}.\\

\begin{definition}\label{esspair}
Let $G$ be a tree. The notation
\[
\overline{v} := (v_{j_1},\ldots,v_{j_i})
\]
will always refer to an $i$-tuple of essential vertices of $G$, appearing in the order induced by the labeling on $G$. We write $\mu(\overline{v})$ to denote the number of connected components of $G-\{v_{j_1},\ldots,v_{j_i}\}$. Given $\overline{v}$, let $\overline{l} := (l_1,\ldots,l_i)$ denote an $i$-tuple of positive integers such that
\[
1 \leq l_k \leq \mu(v_{j_k})-2 \text{ for all $k$}.
\]
We will also set
\[
|\overline{l}| := \sum_k l_k.
\]
Given a pair $(\overline{v},\overline{l})$, we associate to it the summand of $\M_{i,\dt} \otimes_\Z \Q$ generated in degree $2i$ by critical cells $c$ satisfying the following two properties:
\begin{enumerate}
\item If $e \in c$ is an edge, then $\tau(e) = v_{j_k}$ for some $k$, and;
\item if $e \in c$ has $\tau(e) = v_{j_k}$, then there are precisely $l_k$ essential edges which can house witness vertices of $e$.\\
\end{enumerate}
Note that these two pieces of data uniquely determine the edges which are allowed to appear in cells which form the bases for the graded pieces of the summand. Denote this summand by $N^{(\overline{v},\overline{l})}$.\\

We observe that two variables $x_{\mathfrak{e}_1}, x_{\mathfrak{e}_2}$ of $S_G$ act identically on $N^{(\overline{v},\overline{l})}$ if and only if $\mathfrak{e}_1$ and $\mathfrak{e}_2$ are contained in the same connected component of $G - \{v_{j_1},\ldots,v_{j_i}\}$. Let $S^{(\overline{v},\overline{l})}$ denote the quotient of $S_G \otimes_\Z \Q$ by $x_{\mathfrak{e}_2} - x_{\mathfrak{e}_2}$. In particular, $S^{(\overline{v},\overline{l})}$ is isomorphic to a rational polynomial ring in $\mu(\overline{v})$ variables, and $N^{(\overline{v},\overline{l})}$ is a finitely generated graded module over $S^{(\overline{v},\overline{l})}$.\\
\end{definition}

\begin{remark}
If $i = 0$, then $\overline{v}$ and $\overline{l}$ are empty tuples with $\mu(\overline{v}) = 1$. If $i > N_G$, then by convention $N^{(\overline{v},\overline{l})} = 0$.\\
\end{remark}

next goal will be to show that the Hilbert function
\[
n \mapsto \dim_\Q(N^{(\overline{v},\overline{l})})
\]
agrees with a polynomial for all $n$. Note that we already know that this Hilbert function is a polynomial for $n$ sufficiently large, and must only show that this agreement is the case for all $n$. This will imply the same about the Hilbert function for $\M_{i,\dt}$. We accomplish this through the computations in the next section, although we spend time now to set some ground work in this direction. Following this, we will explicitly compute this Hilbert polynomial. We begin with the following.\\

\begin{proposition}\label{monoideal}
The module $N^{(\overline{v},\overline{l})}(i)$ is isomorphic to a squarefree monomial ideal of $S^{(\overline{v},\overline{l})}$.
\end{proposition}

\begin{proof}
We recall that a monomial ideal is called squarefree, if it is generated by monomials, none of which are divisible by a square.\\

By definition, $N^{(\overline{v},\overline{l})}(i)_{i} = N^{(\overline{v},\overline{l})}_{2i}$ is spanned by critical cells containing the edges $e_1,\ldots,e_i$ determined by the pair $(\overline{v},\overline{l})$ (see Definition \ref{esspair}), along with a single witness vertex for each. To define a map $\phi:N^{(\overline{v},\overline{l})}(i) \rightarrow S^{(\overline{v},\overline{l})}$, it suffices to specify where each of these cells is mapped. We set for any such cell $c$
\begin{eqnarray}
\phi(c) = \prod_{\mathfrak{e}}x_{\mathfrak{e}}, \label{monodecomp}
\end{eqnarray}
where the product is over essential edges containing a witness vertex of $c$. Note that it is impossible for a single essential edge to house a witness vertex for multiple edges, and so the above monomial is indeed squarefree. Extending this map through the action of $S^{(\overline{v},\overline{l})}$ defines our desired isomorphism.\\
\end{proof}

Squarefree monomial ideals are some of the most well understood objects in commutative algebra. They are also the subject of the so-called Stanley-Reisner theory (see \cite{MS} for a comprehensive reference on the subject).\\

\subsection{Computing the Hilbert polynomial}\label{hilbertcomp}

We note that for $i = 1$, the polynomial describing the Betti number $b_1(n)$ has been explicitly computed in terms of invariants of the tree $G$. Indeed, this follows from the structure theorem of Ko, and Park \cite[Theorem 3.6]{KP}, and also from the work of Farley and Sabalka \cite{FS}. The Hilbert polynomial has also been computed for all $i$ in the case where $G$ is a tree with maximal degree 3 by Farley \cite{Fa}. Aside from these cases, no explicit description of the polynomial is known. That being said, the work in this paper implies that computing this polynomial is no more difficult than computing the Hilbert polynomials of some squarefree monomial ideals. In other words, for any tree $G$, the work of this paper reduces the task of computing $H_i(B_nG)$ for all $n$ to a finite computation. We will proceed with this computation in this section.\\

\begin{remark}
Proposition \ref{monoideal} reveals that the monomial ideals which appear as graded shifts of summands of $\mathcal{H}_i$ are quite simple. Because of this, we will find that the computation of the Hilbert polynomial of $\mathcal{H}_i$ falls more in the realm of enumerative combinatorics rather than Stanley-Reisner theory. It is likely that studying the cases of more general graphs will require more robust commutative algebra, and we therefore proceed with the case that $G$ is a tree using that language. Our hope is that once the case of general graphs is completed, this will aid in creating a uniform means of approaching asymptotic behaviors in the configuration spaces of graphs.\\
\end{remark}

In this section we will use the decomposition of Proposition \ref{monoideal} to compute the Hilbert polynomial of $\M_{i,\dt}$. Proposition \ref{monoideal} implies that $N^{(\overline{v},\overline{l})}(i)$ is isomorphic to a squarefree monomial ideal of $S^{(\overline{v},\overline{l})}$. To ease computations, we reorder the variables of $S^{(\overline{v},\overline{l})}$ to satisfy the following. Begin by labeling the essential edges adjacent to $v_{j_1}$ which house witness vertices by $x_1$ through $x_{l_1}$, chosen in the order induced by our ordering of the vertices of $G$. Next, label the essential edges housing a witness vertex adjacent to $v_{j_2}$ with the variables $x_{l_1 + 1}$ to $x_{l_1 + l_2}$, using the same rule. Continue in this fashion until $|\overline{l}|$ essential edges have been labeled. We then use the remaining variables to label the final $\mu(\overline{v}) - |\overline{l}|$ essential edges, chosen in any order. Our first observation is the following.\\

\begin{lemma} \label{tensordecomp}
There is an isomorphism of $S^{(\overline{v},\overline{l})}$-modules,
\[
N^{(\overline{v},\overline{l})}(i) \cong (x_1,\ldots,x_{l_1}) \otimes_\Q (x_{l_1 + 1},\ldots, x_{l_1 + l_2}) \otimes_\Q \ldots \otimes_\Q (x_{\sum_{j=1}^{i-1} l_j + 1}, \ldots, x_{|\overline{l}|}) \otimes_\Q \Q[x_{|\overline{l}| + 1},\ldots, x_{\mu(\overline{v})}],
\]
where $(x_1,\ldots,x_{l_1})$ is considered as a module over $\Q[x_1,\ldots,x_{l_1}]$, $(x_{l_1 + 1},\ldots, x_{l_1 + l_2})$ a module over $\Q[x_{l_1+1},\ldots,x_{l_1+l_2}]$, and so on.\\
\end{lemma}

\begin{proof}
Follows immediately from the isomorphism (\ref{monodecomp}).\\
\end{proof}

\begin{definition}
The \textbf{Hilbert series} associated to the module $N^{(\overline{v},\overline{l})}$ is the formal power series
\[
H_{(\overline{v},\overline{l})}(t) := \sum_{n \geq 0} (\dim_\Q N^{(\overline{v},\overline{l})}_n)t^n.
\]
\text{}\\
\end{definition}

Our next goal will be to express $H_{(\overline{v},\overline{l})}(t)$ as a rational function in $t$. From here standard enumerative combinatorics will allow us to compute the Hilbert polynomial.\\

\begin{proposition}\label{hilbertseries}
The Hilbert series $H_{(\overline{v},\overline{l})}(t)$ can be expressed as the rational function,
\[
H_{(\overline{v},\overline{l})}(t) = \frac{t^i\prod_j (1-(1-t)^{l_j})}{(1-t)^{\mu(\overline{v})}}.
\]
\text{}\\
\end{proposition}

\begin{proof}
We will prove this using the fact that the Hilbert series is multiplicative over tensor products. It is easy to see that the Hilbert series for the augmentation ideal of a rational polynomial ring in $d$ variables is equal to the rational function
\[
\frac{1}{(1-t)^d} - 1 = \frac{1 - (1-t)^d}{(1-t)^d}.
\]
Lemma \ref{tensordecomp} therefore implies
\[
H_{(\overline{v},\overline{l})}(t) = t^i \left(\prod_j \frac{1-(1-t)^{l_j}}{(1-t)^{l_j}}\right)\left(\frac{1}{(1-t)^{\mu(\overline{v}) - |\overline{l}|}}\right) = \frac{t^i\prod_j (1-(1-t)^{l_j})}{(1-t)^{\mu(\overline{v})}},
\]
where the factor $t^i$ arises from the graded shift of $N^{(\overline{v},\overline{l})}$.\\
\end{proof}

This will allow us to prove our result on the obstruction to the Hilbert polynomial of $\M_{i,\dt}$.\\

\begin{corollary}\label{exactbound}
Let $G$ be a tree, and fix $i \geq 0$. Then the dimension of $\M_{i,n} \otimes_\Z \Q$ agrees with a polynomial for all $n \geq 0$.\\
\end{corollary}

\begin{proof}
It suffices to prove the claim for the summands $N^{(\overline{v},\overline{l})}$. Proposition \ref{hilbertseries} tells us that the Hilbert series of $N^{(\overline{v},\overline{l})}$ takes the form
\[
\frac{t^i \prod_{j} (1-(1-t)^{l_j})}{(1-t)^{\mu(\overline{v})}}.
\]
The proof of Corollary \ref{eulerpolynomial} implies that it suffices to show the total degree of the numerator $t^i \prod_{j} (1-(1-t)^{l_j})$ is strictly smaller than $\mu(\overline{v})$. Namely, we must argue
\[
i + |\overline{l}| < \mu(\overline{v})
\]
This follows from the fact that each essential vertex of $\overline{v}$, say $v_{j_k}$, is adjacent to at least one essential edge which can never house a witness vertex. Namely, the essential edge which is maximal among all essential edges adjacent to $v_{j_k}$ with respect to the ordering of the vertices of $G$. Moreover, the essential edge containing the root can never house a witness vertex for any essential vertex.\\

Put another way,
\[
i + |\overline{l}| \leq i + \sum_k( \mu(v_{j_k})-2) = i + \mu(\overline{v}) - i - 1 = \mu(\overline{v}) - 1.
\]
This concludes the proof.\\
\end{proof}

We have now proven enough to explicitly compute the Hilbert polynomial of $\M_{i,\dt}$, and consequently $H_i(B_nG)$. The following theorem follows from the results of this section, as well as the combinatorics of generating functions.\\

\begin{theorem}\label{explicitpoly}
Let $G$ be a tree, fix $i,r \geq 0$, and for any pair $(\overline{v},\overline{l})$ write
\[
a^{(\overline{v},\overline{l})}_{i,r} := \sum_{\substack{d_1 + \ldots + d_i = r \\ d_j \geq 1}} (-1)^{r+i}\binom{l_1}{d_1}\cdots \binom{l_i}{d_i}.
\]
Also let $P^{\overline{v}}_i(n)$ be the polynomial
\[
P^{\overline{v}}_i(n) = \binom{\mu(\overline{v})-1+n-i}{\mu(\overline{v}) - 1}.
\]
Then for all $n \geq 0$, the dimension of $H_i(B_n,\Q)$ is equal to the polynomial
\begin{eqnarray}
\sum_{(\overline{v},\overline{l})}\sum_{r = i}^{|\overline{l}|} a^{(\overline{v},\overline{l})}_{i,r}P^{\overline{v}}_i(n-r).\label{explicitform}
\end{eqnarray}
\text{}\\
\end{theorem}

\begin{remark}
One observes from the above formulation that the polynomial $P^{\overline{v}}_i(n)$, as well as the constants $a^{(\overline{v},\overline{l})}_{i,r}$, only depend on the degree sequence of the graph $G$. It follows that the homologies of the braid group of a tree only depends on its degree sequence, as we recorded in Corollary \ref{degreesequence}
\end{remark}

Specializing Theorem \ref{explicitpoly} to the case wherein $i = 1$ yields a much simpler form.\\

\begin{corollary}
Let $G$ be a tree. Then for all $n \geq 0$, the dimension of $H_1(B_nG,\Q)$ agrees with the polynomial
\[
\sum_{v \text{ essential}}\sum_{l=1}^{\mu(v)-2} \binom{\mu(v) - 2 + n}{\mu(v) - 1} - \binom{\mu(v) - 2 + n - l}{\mu(v)-1-l}.
\]
\text{}\\
\end{corollary}

\begin{proof}
In this case we note that for an essential vertex $v$, and a chosen integer $0 \leq l \leq \mu(v)-2$, we have
\[
a^{v,l}_{1,r} = (-1)^{r+1}\binom{l}{r}.
\]
Therefore the polynomial (\ref{explicitform}) can be written
\[
\sum_{v \text{ essential}}\sum_{l=1}^{\mu(v)-2} \sum_{r = 1}^l (-1)^{r+1}\binom{l}{r}\binom{\mu(v)-2+n-r}{\mu(v) - 1}.
\]
Using the principle of inclusion-exclusion, it can be seen that
\[
\sum_{r = 1}^l (-1)^{r+1}\binom{l}{r}\binom{\mu(v)-2+n-r}{\mu(v) - 1} = \binom{\mu(v) - 2 + n}{\mu(v) - 1} - \binom{\mu(v) - 2 + n - l}{\mu(v)-1-l},
\]
as desired.\\
\end{proof}

A result of this kind appears in the work of Farley and Sabalka, where they compute the number of generators of the group $B_nG$ \cite{FS}. A formula for the case $i = 1$ is also found in the results of Ko and Park \cite{KP}.\\

Another simple case is that in which every essential vertex of $G$ has degree 3. In this case Theorem \ref{explicitpoly} can be used to recover the following result of Farley \cite{Fa}.\\

\begin{corollary}
Let $G$ be a tree whose every essential vertex has degree 3. Then for all $n \geq 0$, the dimension of $H_i(B_nG;\Q)$ agrees with the polynomial
\[
\binom{N_G}{i}\binom{n}{2i}
\]
\text{}\\
\end{corollary}

\begin{proof}
The major thing to note in this case is that for any $\overline{v}$, the only associated vector $\overline{l}$ is $(1,\ldots,1)$. Therefore,
\[
a^{(\overline{v},\overline{l})}_{i,r} = \begin{cases} 1 &\text{ if $r = i$}\\ 0 &\text{ otherwise.}\end{cases}
\]
The formula (\ref{explicitform}) now becomes
\[
\binom{N_G}{i}P^{\overline{v}}_i(n-i) = \binom{N_G}{i}\binom{\mu(\overline{v})-1+n-2i}{\mu(\overline{v}) - 1} = \binom{N_G}{i}\binom{2i+1-1+n-2i}{2i} = \binom{N_G}{i}\binom{n}{2i}.
\]
This completes the proof.\\
\end{proof}

\section{Concluding remarks}

In this section we consider what this work implies for the homology of graph braid groups when $G$ is a general graph rather than a tree.\\

Let $G$ be a graph with $N_G$ essential vertices, and $E$ essential edges. Also assume that $G$ is homeomorphic to neither $S^1$ nor a line segment. Then the work of Section \ref{morsecomplexismodule} implies that $\M_{i,\dt}$ carries the structure of a finitely generated $S_G$-module. If $G$ is not a tree, however, it is no longer the case that the Morse differential is trivial. We therefore pose the following question.

\begin{question}
Does the action of $S_G$ on $\M_{i,\dt}$ commute with the Morse differential? In other words, is the Morse differential a morphism of graded $S_G$-modules?
\end{question}

We note that altering our choice of spanning tree $T$ has a non-trivial influence on the Morse differential. It is therefore unclear whether the Morse differential will commute with the action of $S_G$ for all choices of $T$. We therefore refine Question 1 as follows.

\begin{question}
Does there exist a choice of spanning tree $T$ of $G$ such that the action of $S_G$ on $\M_{i,\dt}$ commutes with the Morse differential?
\end{question}

If the answer to these questions is affirmative, then we can immediately conclude many things about the asymptotic behavior of these homology groups. Indeed, in this case the Hilbert basis theorem would imply that there is a finitely generated $S_G$-module $\mathcal{H}_i$ such that $(\mathcal{H}_i)_n \cong H_i(B_nG)$.\\

\begin{corollary}\label{maincor}
Assume that Questions 1 and 2 have an affirmative answer. For each $i,n$ write
\[
H_i(B_nG) \cong \Z^{b_i(n)} \oplus \bigoplus_{p \text{ prime}}\bigoplus_{j \geq 1} (\Z/p^j\Z)^{m_{p,i,j}(n)}.
\]
Then:
\begin{enumerate}
\item there exists a polynomial $P_i(t) \in \Q[t]$ of degree $\leq E-1$ such that $P_i(n) = b_i(n)$ for all $n \gg 0$;
\item there exists a positive integer $e_i$, independent of $n$, such that the exponent of $H_i(B_nG)$ is at most $e_i$;
\item for each prime $p$ and each positive integer $j$, there is a polynomial $P_{i,p^j}(t) \in \Q[t]$ such that $P_{i,p^j}(n) = m_{p,i,j}(n)$ for $n \gg 0$. In particular, there is a finite set of primes $p_{i,1},\ldots,p_{i,l}$, and polynomials $m_{i,p_j}(n)$ such that the order of the torsion subgroup of $H_i(B_nG)$ is precisely $p_1^{m_{p_{i,1}}(n)}\cdots p_l^{m_{p_{i,l}}(n)}$.\\
\end{enumerate}
\end{corollary}

\begin{proof}
Follows from standard facts in the study of graded modules over integral polynomial rings.\\
\end{proof}

The work of Ko and Park \cite[Theorem 3.6]{KP} that $H_1(B_nG)$ has a torsion subgroup annihilated by 2, and that the multiplicity of this torsion is eventually constant in $n$. Their work also implies that the rank of $H_1(B_nG)$ will be constant whenever $G$ is sufficiently connected. In contrast to this, affirmative answers to Questions 1 and 2 imply that the top homology of $B_nG$ always grows as fast as possible.\\

\begin{theorem}\label{tophomology}
Assume that Questions 1 and 2 have an affirmative answer. Then the Hilbert function
\[
n \mapsto \dim_\Q(H_{N_G}(B_nG;\Q))
\]
is equal to a polynomial of degree $E-1$ for $n \gg 0$.
\end{theorem}

\begin{proof}
Theorem \ref{eulerpolynomial} implies that the Euler characteristic of $\UConf_n(G)$ is a polynomial of degree $E-1$. On the other hand, Remark \ref{slowgrowth} implies the lower homology groups have Hilbert functions which grow strictly slower than this. It follows that the rank of $H_{N_G}(B_nG)$ must eventually grow like a polynomial of degree $E-1$.\\
\end{proof}

Note that Theorem \ref{morsediffcomp} implies that $H_{N_G}(B_nG)$ is torsion-free. We have therefore seen that $H_1(B_nG)$ can only have 2 torsion, while the top homology $H_{N_G}(B_nG)$ does not have any torsion. It is an interesting question to ask what other types of torsion can appear in the intermediate homologies.\\

While an affirmative answer to Questions 1 and 2 would be a step in the right direction, the work of Ko and Park \cite{KP} would suggest that the action of $S_G$ on $\M_{i,\dt}$ described in Section \ref{morsecomplexismodule} is perhaps non-optimal. We have already discussed that the result \cite[Theorem 3.6]{KP} implies that the braid groups of biconnected graphs have eventually constant first homologies. However, the action of $S_G$ defined in Section \ref{morsecomplexismodule} only detects the connectivity a chosen spanning tree of $G$. In particular, is it possible that one can improve on this action by accounting for connectivity granted by the deleted edges?

\begin{question}
Let $G$ be a graph which is neither a circle, nor a line segment. Does there exist an action of $S_G$ on $\M_{i,\dt}$ which interacts in a non-trivial way with the deleted edges of $G$? More specifically, can one define this action so that if $c$ is a critical $i$-cell, then there is a $\Z$-linear relation between $x_{\mathfrak{e}_1}\cdot c$ and $x_{\mathfrak{e}_2} \cdot c$ whenever the two essential edges $\mathfrak{e}_1$ and $\mathfrak{e}_2$ are in the same connected component of $G$ with the edges of $c$ removed?
\end{question}

The existence of such an action would actually suggest something quite strong about what kind of invariants of $G$ are encoded by the homology of its braid groups. If $G$ is a graph, recall that we define $\Delta_G^i$ to be the maximum number of connected components $G$ can be broken into by removing exactly $i$-vertices.\\

\begin{conjecture}\label{mainconj}
Let $G$ be a graph which is neither a line segment, nor a circle. Then for all $i$, and all $n \gg i$, the Betti number
\[
b_i(n) := \dim_\Q(H_i(B_nG;\Q))
\]
agrees with a polynomial of degree $\leq \Delta^i_G-1$.\\
\end{conjecture}

Note that Theorem \ref{tophomology} is a specific case of this conjecture. Also note, if $G$ is biconnected, then the above implies that $b_1(n)$ is eventually constant. This fact is proven in the work of Ko and Park \cite[Theorem 3.6]{KP}. The above conjecture therefore proposes a generalization of one aspect of that work.\\

Finally, we think about the ramifications of this work on the spaces $\UConf_n(G)$. The action of $S_G$ on $\M_{i,\dt}$ is defined in a purely formal way. However, in the study of more classical configuration spaces, stability phenomena often arise from maps on the underlying spaces $\UConf_n(G) \rightarrow \UConf_{n+1}(G)$ (see, for instance, \cite{EW-G} and \cite{McD}). Is that also the case here?\\

\begin{question}
If $G$ is a tree, is the action of $S_G$ on $\M_{i,\dt}$ induced by a map of topological spaces $\UConf_n(G) \rightarrow \UConf_{n+1}(G)$? If $G$ is a general graph, and Questions 1 and 2 have affirmative answers, is the action of $S_G$ on the homologies of $B_nG$ induced in a similar way?\\
\end{question}

Of course, Question 4 is highly relevant for Questions 1 and 2. If one understands the action of $S_G$ topologically, then it might make the existence of such an action on the homologies of general graph braid groups more apparent.\\


\begin{thebibliography}{aaaa}
\small
\bibitem[A]{A} A. Abrams, \textit{Configuration spaces and braid groups of graphs}, Ph.D thesis, \url{home.wlu.edu/~abramsa/publications/thesis.ps}.
\bibitem[BF]{BF} K. Barnett and M. Farber, \textit{Topology of configuration space of two particles on a graph, I}, Algebr.
Geom. Topol. 9(1) (2009), 593–624. \arXiv{0903.2180}.
\bibitem[CEF]{CEF} T. Church, J. S. Ellenberg and B. Farb, \textit{$\FI$-modules and stability for representations of symmetric groups}, Duke Math. J. 164, no. 9 (2015), 1833-1910.
\bibitem[EW-G]{EW-G} J. S. Ellenberg and J. D. Wiltshire-Gordon, \textit{Algebraic structures on cohomology of configuration spaces of manifolds with flows}, \arXiv{1508.02430}.
\bibitem[Fa]{Fa} D. Farley, \textit{Homology of tree braid groups}, Topological and asymptotic aspects of group theory, 101-112, Contemp. Math., 394, Amer. Math. Soc., Providence, RI, 2006. \url{http://www.users.miamioh.edu/farleyds/grghom.pdf}.
\bibitem[Fo]{Fo} R. Forman, \textit{Morse theory for cell complexes}, Adv. in Math. 134 (1998), pp. 90-145.
\bibitem[Fo2]{Fo2} R. Forman, \textit{A user's guide to discrete Morse theory}, S\'eminaire Lotharingien de Combinatoire 48 (2002), 35 p. \url{http://www.emis.de/journals/SLC/wpapers/s48forman.pdf}.
\bibitem[FS]{FS} D. Farley and L. Sabalka, \textit{Discrete Morse theory and graph braid groups}, Algebr. Geom. Topol. 5 (2005), 1075-1109 (electronic). \url{http://www.users.miamioh.edu/farleyds/FS1.pdf}.
\bibitem[FS2]{FS2} D. Farley and L. Sabalka, \textit{Presentations of graph braid groups}, Forum Math. 24 (2012), no. 4, 827-859. \url{http://www.users.miamioh.edu/farleyds/FS3final.pdf}.
\bibitem[Gh]{Gh} R. Ghrist, \textit{Configuration spaces and braid groups on graphs in robotics}, Knots, braids, and mapping class groups - papers dedicated to Joan S. Birman (New York, 1998), AMS/IP Stud. Adv. Math., 24, Amer. Math. Soc., Providence, RI (2001), 29–40. \url{https://www.math.upenn.edu/~ghrist/preprints/birman.pdf}.
\bibitem[Ga]{Ga} S. R. Gal, \textit{Euler characteristic of the configuration space of a complex}, Colloq. Math. 89 (2001), 61-67. \url{http://arxiv.org/pdf/math/0202143v1.pdf}.
\bibitem[KKP]{KKP} J. H. Kim, K. H. Ko, and H. W. Park, \textit{Graph braid groups and right-angled Artin groups}, Trans. Amer. Math. Soc. 364 (2012), 309-360. \arXiv{0805.0082}.
\bibitem[KP]{KP} K. H. Ko, and H. W. Park, \textit{Characteristics of graph braid groups}, Discrete Comput Geom (2012) 48: 915. \arXiv{1101.2648}.
\bibitem[L]{L} D. L\"utgehetmann, \textit{Configuration spaaces of graphs}, Masters Thesis, \url{http://luetge.userpage.fu-berlin.de/pdfs/masters-thesis-luetgehetmann.pdf}.
\bibitem[McD]{McD} D. McDuff, \textit{Configuration spaces of positive and negative particles}, Topology, Volume 14, Issue 1, March 1975, Pages 91-107.
\bibitem[MS]{MS} E. Miller and B. Sturmfels, \textit{Combinatorial Commutative Algebra}, Graduate Texts in Mathematics \# 227, Springer-Verlag New York, 2005.
\bibitem[PS]{PS} P. Prue and T. Scrimshaw, \textit{Abrams’s stable equivalence for graph braid groups}, \arXiv{0909.5511}.
\bibitem[SS]{SS} S. Sam, and A. Snowden, \textit{Proof of Stembridge's conjecture on stability of Kronecker coefficients}, J. Algebraic Combin. 43 (2016), no. 1, 1-10, \arXiv{1501.00333}.
\end{thebibliography}
\end{document}